\documentclass[12pt]{amsart}
\usepackage[colorlinks=true,linkcolor=blue,anchorcolor=blue,backref=page,linktocpage=true,urlcolor=blue,citecolor=blue]{hyperref}
\usepackage{amssymb}
\usepackage{mathrsfs}
\usepackage{enumerate}
\usepackage{hyperref}
\usepackage{bbm}
\usepackage{color}

\newtheorem{theorem}{Theorem}[section]
\newtheorem{definition}{Definition}[section]
\newtheorem{proposition}[theorem]{Proposition}
\newtheorem{lemma}[theorem]{Lemma}
\newtheorem{corollary}[theorem]{Corollary}
\newtheorem{claim}{Claim}

\newtheorem*{question}{Question}

 \def\NN{{\mathbb N}} 

 \def\RR{{\mathbb R}}

\def \G{\Gamma}
\def \D{\Delta}

\def\G{{\mathcal G}}

\begin{document}

\title[Stable manifolds]{$C^1$ Pesin (un)stable manifolds without domination}

\author[Yongluo Cao]{Yongluo Cao}

\address{Department of Mathematics, Soochow University, Suzhou 215006, Jiangsu, China.}
\address{Center for Dynamical Systems and Differential Equation, Soochow University, Suzhou 215006, Jiangsu, China}

\email{\href{mailto:ylcao@suda.edu.cn}{ylcao@suda.edu.cn}}
%
\author[Zeya Mi]{Zeya Mi}

\address{School of Mathematics and Statistics,
Nanjing University of Information Science and Technology, Nanjing 210044, China}
\email{\href{mailto:mizeya@163.com}{mizeya@163.com}}

\author[Rui Zou]{Rui Zou}
\address{School of Mathematics and Statistics,
Nanjing University of Information Science and Technology, Nanjing 210044, China}
\email{\href{zourui@nuist.edu.cn}{zourui@nuist.edu.cn}}
%
%
\thanks{Z.Mi is the corresponding author. We are partially supported by National Key R\&D Program of China(2022YFA1007800). Y. Cao was partially supported by NSFC 12371194; Z. Mi was partially supported by NSFC 12271260; R. Zou was partially supported by NSFC (12471185, 12271386). }


\date{\today}


\begin{abstract}
For $C^1$ diffeomorphisms with continuous invariant splitting without domination, we prove the existence of  (un)stable manifolds under the hyperbolicity of invariant measures.
\end{abstract}

\maketitle


\section{Introduction}

The stable manifold theorem is one of the key tools in studying hyperbolic behavior on dynamical systems. It was first established in $C^1$ uniform hyperbolic systems (see e.g. \cite{HPS77, BS}), 
which asserts the existence of stable/unstable manifolds tangent to the uniformly contracted/expanded sub-bundles, respectively. 
This has been developed in $C^{1+\alpha}(\alpha>0)$ non-uniformly hyperbolic systems by Pesin and others (see e.g. \cite{FY,K80,P76,P77,R79}), known as the Pesin's stable manifold theorem. 
It is important to investigate whether the $C^{1+\alpha}$ regularity assumption is essential or not.

Pugh \cite{P84} built a  $C^1$ diffeomorphism which gives a counterexample to Pesin's stable manifold theorem.
More recently, Bonatti-Crovisier-Shinohara \cite{BC13} showed that the non-existence of (un)stable manifolds is a generic phenomenon in the $C^1$ category.
However, under the existence of domination on Oseledec's splitting of the hyperbolic (without zero Lyapunov exponents) measure, one knows that the Pesin's stable manifold holds true. This was announced by Ma\~n\'e \cite[Page 1271]{M84} in his 1982 ICM's report, and was made precise by Abdenur-Bonatti-Crovisier\cite[\S 8]{ABC} and also by Avila-Bochi \cite[Theorem 4.7]{ab12} with different approaches.


\medskip
In the present paper, we remove the domination assumption in \cite{ABC, ab12}, and show that the Pesin's stable manifold theorem remains valid for $C^1$ diffeomorphisms admitting continuous invariant splitting of hyperbolic measures.

\medskip

Let $M$ be a compact Riemannian manifold without boundary. Let $d$ denote the distance on $M$ induced by its Riemannian metric. Denote by ${\rm Diff}^1(M)$ the space of $C^1$ diffeomorphisms endowed with the usual $C^1$-topology. Let $f\in {\rm Diff}^1(M)$ and $\mu$ be an $f$-invariant Borel probability measure. By Oseledec's theorem \cite{O}, 
for $\mu$-almost every point $x \in M$, for every $v\in T_xM\setminus \{0\}$, there exists the limit
$$
\chi(x,v)=\lim_{n\to \pm \infty}\frac{1}{n}\log \|Df^n(x)v\|,
$$
called the Lyapunov exponents of $x$ at $v$.

A   stable set at a point $x\in M$ is given by
$$
W^s(x)=\left\{y\in M: \lim_{n\to \infty}d(f^nx,f^ny)=0\right\},
$$
and analogously for the unstable set $W^u(x)$, replace $f$ by $f^{-1}$. We call $W^s(x)$/$W^u(x)$ a (global) \emph{  stable/unstable manifold} of $x$ if it is an injectively immersed sub-manifold.
The Pesin's stable manifold theorem asserts that for any hyperbolic invariant measure,  for $\mu$-almost every point $x$, there exist a   stable manifold and an   unstable manifold corresponding to negative Lyapunov exponents and positive Lyapunov exponents with complementary dimensions.


Now we state our main result as follows. More general versions will be presented in Section \ref{main}.

\begin{theorem}\label{m}
Let $f\in {\rm Diff}^1(M)$ with a continuous invariant splitting $TM=E\oplus F$. 
Assume that $\mu$ is an invariant measure whose Lyapunov exponents along $E$ are all positive and along $F$ are all negative.

Then, for $\mu$-almost every point $x$, there exist an unstable manifold and a   stable manifold tangent to $E$ and $F$ respectively.
\end{theorem}

%
%
%

The Pesin's stable manifold theorem was carried out by Pliss \cite{Pliss} for $C^1$ diffeomorphisms when all Lyapunov exponents are strictly negative, using the celebrated Pliss's argument on hyperbolic times. It was then extended in \cite{ABC, ab12} to the case that the measure $\mu$ has positive and negative Lyapunov exponents simultaneously, whose Oseledec splitting $E\oplus F$ is \emph{uniformly dominated}, i.e., there are $C>0$ and $\lambda\in (0,1)$ such that for $\mu$-almost every $x\in M$ and each $n\in \NN$,
$$
\|Df^{-n}|_{E(f^n(x))}\| \cdot \|Df^n|_{F(x)}\|\le C\lambda^n.
$$
Theorem \ref{m} indicates that Pesin's stable manifold theorem does hold for diffeomorphisms which are only $C^1$, as long as the $C^{1+\alpha}$ hypothesis is replaced by a continuous invariant splitting hypothesis on the measure's Oseledec splitting. It extends the work of \cite{ABC,ab12} to $C^1$ diffeomorphisms without domination.

When the measure possesses both positive and negative Lyapunov exponents, different to \cite{Pliss}, one needs the domination on the measure's Oseledec splitting to control the geometry of iterated disks tangent to stable/unstable directions. Both the original $C^{1+\alpha}$  hypothesis and the uniform domination hypothesis in \cite{ABC,ab12} provide such a domination. In particular, the main tools in \cite{ABC,ab12} are the Plaque family theorem \cite{HPS77} and the cone argument, which are guaranteed by uniform domination. 

The main ingredient towards Theorem \ref{m} is to create domination. A weak version of domination called averaged domination arises from the gap of Lyapunov exponents on sub-bundles. Though, weaker than uniform domination, it ensures that one can control the geometry well of iterated disks at hyperbolic times. This will be sufficient to obtain the unstable manifolds by considering the limit of these iterated disks.

\section{Unstable manifolds from averaged domination}

Throughout this paper, let $M$ be a compact Riemannian manifold without boundary, and $f$ be a $C^1$ diffeomorphism on $M$. 

Given $x\in M$, a $C^1$-smooth embedded disk $W$ centered at $x$ is called a \emph{local stable manifold} at $x$ if there exist $C>0$ and $\lambda\in (0,1)$ such that for any $n\ge 1$, 
$$
d(f^n(y),f ^n(z))\le C\lambda^nd(y,z)
$$ 
whenever $y,z\in W$. Replacing $n$ with $-n$, one can define the \emph{local unstable manifold} similarly.


Denote by $m(A)=\inf_{\|v\|=1}\|Av\|=\|A^{-1}\|^{-1}$ the \emph{mini-norm} of linear isomorphism $A$.
Let $E$ and $F$ be  two sub-bundles. For $\lambda>1$, denote
$$
\D_{n}(\lambda,f,E,F)=\Big\{x\in M: \prod_{j=0}^{k-1}\frac{m\left(Df|_{E(f^jx)}\right)}{\left\|Df|_{F(f^jx)}\right\|}\ge \lambda^k, \quad \forall\, 1\le k \le n\Big\}.
$$
When $x\in \D_{n}(\lambda,f)$, we say that $(x,f^n(x))$ is $\lambda$-\emph{average dominated} associated to $(E,F)$.
We will also study the set 
\begin{equation}\label{def:delta}
\D(\lambda,f,E,F)=\bigcap_{n\ge 1}\D_{n}(\lambda,f,E,F).
\end{equation}

\begin{definition}\label{hdx}
Let $f\in {\rm Diff}^1(M)$ and $E, F$ be two continuous invariant sub-bundles of $TM$.
Given $\lambda_1>1,\lambda_2>1$, denote by $$\mathcal{HD}(x,\lambda_1,\lambda_2, f, E, F)$$ the set of times $n$ such that 
$x\in \D_{n}(\lambda_2,f,E,F)$ and $n$ is a $(\lambda_1,E)$-hyperbolic time for $x$ with respect to $f$, i.e.,
$$
\prod_{j=n-k}^{n-1}m\left(Df|_{E(f^jx)}\right)\ge \lambda_1^k, \quad 1\le k \le n.
$$
\end{definition}

We also write $\mathcal{HD}(x,\lambda_1,\lambda_2)$ when there is no ambiguity.

Let $f$ be a $C^1$ diffeomorphism with a continuous invariant splitting $TM=E\oplus F$.
Given $a\ge 0$, for every $x\in M$, define the $E$-direction \emph{cone of width} $a$ at $x$ as follows:
$$
\mathscr{C}^E_{a}(x)=\left\{v=v^E\oplus v^F\in E(x)\oplus F(x): \|v_F\|\le a \|v_E\|\right\}.
$$
We say a $C^1$ embedded sub-manifold $D$ is \emph{tangent to} $\mathscr{C}^E_{a}$ if it has dimension ${\rm dim}E$ and $T_xD\subset \mathscr{C}^E_{a}(x)$ for every $x\in D$. In particular, we will say $D$ is \emph{tangent to} $E$ everywhere if it is \emph{tangent to} $\mathscr{C}^E_{0}$. 

By applying the averaged domination together with the backward contraction property on hyperbolic times, one can conclude the following result for sub-manifold tangent to the cone-field. See \cite[Lemma 3.7, Proposition 3.8]{MCY18} for the proof.

\begin{lemma}\label{aver}
Let $f$ be a $C^1$ diffeomorphism with a continuous invariant splitting $TM=E\oplus F$.

For any $\lambda_1>\sigma_1>1,\lambda_2>\sigma_2>1$, there exist $a>0$ and $r>0$ such that if $n\in \mathcal{HD}(x,\lambda_1,\lambda_2)$ for some $x\in M$, then for any sub-manifold $D$ tangent to $\mathscr{C}^E_{a}$ with radius larger than $r$ around $x$, for every $y \in D$ satisfying $d(f^n (x), f^n (y))\le r$ we have
\begin{itemize}
\smallskip
\item $d (f^{k} (x), f^{k} (y) ) \le \sigma_1^{-n+k} d (f^n (x), f^n (y))$ for any $0 \le k \le n$;
\smallskip
\item
$d_{\mathcal{G}}(T_{f^k(y)}f^k(D), E(f^k(y))) \le  \sigma_2^{-k} \cdot a$ for any $0\le k \le n$.
\end{itemize}
Here $d_{\mathcal{G}}$ denotes the distance in the Grassmannian manifold of ${\rm dim}E$-dimensional subspaces in $TM$.
\end{lemma}

Given a subset $A\subset \NN$, we define its \emph{lower density} and \emph{upper density} as follows:
$$
D_{L}(A):=\liminf_{n\to +\infty}\frac{\# (A\cap \{0,\cdots n-1\})}{n},
$$
$$
D_{U}(A):=\limsup_{n\to +\infty}\frac{\# (A\cap \{0,\cdots n-1\})}{n}.
$$

For an invariant measure $\mu$, we say a point $x\in M$ is $\mu$-\emph{generic} if 
$$
\frac{1}{n}\sum_{i=0}^{n-1}\delta_{f^i(x)}\to \mu \quad \textrm{as}~n\to +\infty.
$$
If $\mu$ is ergodic, by Birkhoff's ergodic theorem \cite[Theorem 1.14]{wa82}, we know that $\mu$-almost every point $x$ is $\mu$-generic. 

The next result provides a condition on some $\mu$-generic point for the ergodic measure that guarantees the existence of unstable manifolds almost everywhere.


\begin{theorem}\label{mg}
Let $f\in {\rm Diff}^1(M)$ with a continuous invariant splitting $TM=E\oplus F$.

If $\mu$ is an ergodic measure for which there exist $\lambda_1>1,\lambda_2>1$ and a $\mu$-generic point $x$ such that 
\begin{equation}\label{denassum}
D_U(\mathcal{HD}(x,\lambda_1,\lambda_2)>0,
\end{equation}
then, at $\mu$-almost every point $x$ there exists an   unstable manifold $W^{E,u}(x)$ tangent to $E$. 

Moreover, for every $\chi\in (1, \lambda_1)$ there exist Borel functions $T(x)$ and $\delta(x)$ such that for $\mu$-almost every $x\in M$, there is a local   unstable manifold $W^{E,u}_{\delta(x)}(x)$ of size $\delta(x)$ contained in the global   unstable manifold $W^{E,u}(x)$ such that for any $y\in W^{E,u}_{\delta(x)}(x)$ and $n\in \NN$,
$$
d(f^{-n}(x), f^{-n}(y))\le T(x)\chi^{-n}d(x,y).
$$
\end{theorem}

\begin{proof}
Let $\mu$ be an ergodic measure and take $x$ as $\mu$-generic point satisfying \eqref{denassum} for constants $\lambda_1>1,\lambda_2>1$.
Thus, there exist $\theta\in (0,1)$ and a subsequence $\{n_k\}_{k\in \NN}$ such that 
\begin{equation}\label{kk}
\frac{1}{n_k}\#\left(\mathcal{HD}(x,\lambda_1,\lambda_2)\cap\{0,\cdots,n_k-1\}\right)>\theta,\quad \forall\, k\in \NN.
\end{equation}
Write for simplicity the measures
$$
\mu_{n_k}=\frac{1}{n_k}\sum_{j=0}^{n_k-1}\delta_{f^j(x)}, \quad k\in \NN.
$$
Hence, we have $\mu_{n_k}\to \mu$ as $k\to +\infty$ by definition.
For each $j\ge 0$, consider 
$$
H_j=H_j(x)=
\begin{cases}
\{f^j(x)\}, & j\in \mathcal{HD}(x,\lambda_1,\lambda_2);\\
 ~~~~\emptyset,&~~{\rm otherwise}.
\end{cases}
$$
Observe that 
$$
\frac{1}{n_k}\sum_{j=0}^{n_k-1}\delta_{f^j(x)}(H_j)=\frac{1}{n_k}\#\left(\mathcal{HD}(x,\lambda_1,\lambda_2)\cap \{0,\cdots,n_k-1\}\right).
$$
According to (\ref{kk}), one gets 
\begin{equation}\label{getheta}
\mu_{n_k}\left(\bigcup_{0\le j \le n_k-1}H_j\right)\ge \frac{1}{n_k}\sum_{j=0}^{n_k-1}\delta_{f^j(x)}(H_j)>\theta,\quad \forall \, k\in \NN.
\end{equation}
Consider $H_{\infty}$ as the limit set of $\{H_j\}_{j\ge 0}$, that is 
$$
H_{\infty}=\bigcap_{m\ge 1}\overline{\bigcup_{j\ge m}H_j}.
$$
We have the following claim:

\begin{claim}\label{theta}
$$
\mu(H_{\infty})\ge \limsup_{k\to +\infty}\mu_{n_k}\left(\bigcup_{0\le j \le n_k-1}H_j\right)\ge \theta.
$$
\end{claim}
\begin{proof}
Since $\{\overline{\cup_{j\ge m}H_j}\}_{m\in \NN}$ is decreasing, we get
\begin{equation}\label{ineq2}
\mu(H_{\infty})=\lim_{m\to +\infty}\mu\left(\overline{\bigcup_{j\ge m}H_j}\right).
\end{equation}
By the convergence of $\mu_{n_k}\to \mu$, for every $m\in \NN$ we see that
\begin{equation}\label{t}
\mu\left(\overline{\bigcup_{j\ge m}H_j}\right)\ge \limsup_{k\to +\infty}\mu_{n_k}\left(\overline{\bigcup_{j\ge m}H_j}\right)\ge \limsup_{k\to +\infty}\mu_{n_k}\left(\bigcup_{j\ge m}H_j\right).
\end{equation}
Observe that for every $k\in \NN$ and every $m\in\NN$,
$$
\mu_{n_k}\left(\bigcup_{0\le j \le n_k-1}H_j\right)\le \mu_{n_k}\left(\bigcup_{0\le j \le m-1}H_j\right)+\mu_{n_k}\left(\bigcup_{j\ge m }H_j\right).
$$
The definition of $\{H_j\}_{j\ge 0}$ gives that
$$
\lim_{k\to +\infty}\mu_{n_k}\left(\bigcup_{0\le j \le m-1}H_j\right)=0,\quad \forall \, m\in \NN.
$$
Thus, for every $m\in \NN$ we have 
\begin{equation}\label{ineq1}
\limsup_{k\to +\infty}\mu_{n_k}\left(\bigcup_{j\ge m }H_j\right)\ge \limsup_{k\to +\infty}\mu_{n_k}\left(\bigcup_{0\le j \le n_k-1}H_j\right).
\end{equation}
Combining \eqref{ineq2}, \eqref{t}, \eqref{ineq1} and \eqref{getheta} we conclude that
$$
\mu(H_{\infty})\ge \limsup_{k\to +\infty}\mu_{n_k}\left(\bigcup_{0\le j \le n_k-1}H_j\right)\ge \theta.
$$
\end{proof}

Now we show that there exists a local unstable disk centered at any $w\in H_{\infty}$.
Consider any $\chi \in (1,\lambda_1)$ and $\sigma\in (1,\lambda_2)$.
Let $a>0$, $r>0$ be the constants given by Lemma \ref{aver}. Take $D_0$ as a $C^1$ disk tangent to $\mathscr{C}^E_{a}$ centered at $x$ with radius $r$. From Lemma \ref{aver}, we know that if $n\in \mathcal{HD}(x,\lambda_1,\lambda_2)$, then 
\begin{enumerate}
\item\label{F1} $f^n(D_0)$ contains a ball $D_{n}$ of radius $r$ centered at $f^n(x)$.
\smallskip
\item\label{F2} $d (f^{-k}(y), f^{-k}(z) ) \le \chi^{-k} d(y, z)$ for any $0 \le k \le n$ and $y,z\in D_n$.
\smallskip
\item\label{F3}
$d_{\G}\left(T_{f^{k}(y)}f^{k-n}(D_n), E(f^{k}(y))\right) \le  \sigma^{-k} \cdot a$ for any $0\le k \le n$ and $y\in f^{-n}(D_{n})$.
\end{enumerate}
We claim that:
\begin{claim}\label{hdn}
$\{D_{n}\}_{n\in \mathcal{HD}(x,\lambda_1,\lambda_2)}$ is uniformly equicontinuous in the $C^1$-topology.
\end{claim}

\begin{proof}
Let $\mathcal{HD}(x,\lambda_1,\lambda_2)=\{h_i\}_{i\ge 1}$ with $h_{i+1}>h_{i}$ for every $i\ge 1$. Fix any $\varepsilon>0$. Since $E$ is continuous, there is $\delta_1>0$ such that
\begin{equation}\label{jk}
d_{\G}(E(y),E(z))< \varepsilon/3
\end{equation}
whenever $d(y,z)< \delta_1$.
By the property (\ref{F3}), there exists $i(\varepsilon)\in \NN$ such that for any $i\ge i(\varepsilon)$, one has that 
\begin{equation}\label{mn}
d_{\G}(T_{y}(D_{h_i}), E(y)) \le  \sigma^{-h_i} \cdot a<\varepsilon/3, \quad \forall \, y\in D_{h_i}.
\end{equation}
By (\ref{jk}) and (\ref{mn}), we see that for every $i\ge i(\varepsilon)$ and $y,z\in D_{h_i}$ satisfying $d(y,z)<\delta_1$, 
\begin{eqnarray*}
d_{\G}(T_yD_{h_i}, T_zD_{h_i})&\le& d_{\G}(T_yD_{h_i}, E(y))+d_{\G}(E(y), E(z))\\
&+&d_{\G}(E(z), T_zD_{h_i})\\
&<&\varepsilon.
\end{eqnarray*}
On the other hand, as every $D_{h_i}$ is $C^1$, there is $\delta_2>0$ such that for every $1\le i \le i(\varepsilon)$, if $y,z\in D_{h_i}$ with $d(y,z)<\delta_2$, then
$
d_{\G}(T_yD_{h_i}, T_zD_{h_i})<\varepsilon.
$
Consequently, by taking $\delta=\min\{\delta_1,\delta_2\}$, we get for every $i$ that
$$
d_{\G}(T_yD_{h_i}, T_zD_{h_i})<\varepsilon
$$
whenever $d(y,z)<\delta$.
This shows the desired result. 
\end{proof}

For any $w\in H_{\infty}$, by construction there exists an increasing sequence $1\le t_1<t_2<\cdots<t_n<\cdots$ in $\mathcal{HD}(x,\lambda_1,\lambda_2)$ such that $f^{t_i}(x)\to w$ as $i\to +\infty$. Since $\{D_{n}\}_{n\in \mathcal{HD}(x,\lambda_1,\lambda_2)}$ is uniformly equicontinuous by Claim \ref{hdn}, according to the Ascoli-Arzela theorem, up to considering subsequences, there exists a $C^1$-disk $D^u(w,r)$ centered at $w$ of radius $r$ such that $D_{t_i}$ converges to $D^u(w,r)$ in $C^1$-topology. For any fixed $n\in \NN$, choose $i_0\in \NN$ large enough such that $t_i\ge n$ for every $i\ge i_0$. From the convergence of $D_{t_i}\to D^u(w,r)$, for any $y,z\in D^u(w,r)$, there exist 
$y_{t_i}, z_{t_i}\in D_{t_i}$ such that $y_{t_i}\to y$, $z_{t_i}\to z$. 
By property (\ref{F2}), it follows that 
$$
d\left(f^{-n}(y_{t_i}), f^{-n}(z_{t_i})\right) \le \chi^{-n} d (y_{t_i}, z_{t_i}), \quad \forall\, i\ge i_0.
$$
By passing to the limit as $i\to +\infty$, we get
$$
d(f^{-n}(y), f^{-n}(z) ) \le \chi^{-n} d (y, z).
$$
From the arbitrariness of $n\in \NN$ we know
$$
d(f^{-n}(y), f^{-n}(z) ) \le \chi^{-n} d (y, z),\quad \forall \, n\in\NN.
$$
In view of property (\ref{F3}) above, taking the limit as $i\to +\infty$ we obtain that $D^u(w,r)$ is tangent to $E$.
Writing $W^{E,u}_r(w)=D^u(w,r)$ for every $w\in H_{\infty}$, which is a local   unstable manifold tangent to $E$ of size $r$.

Set 
$$
\Lambda_0=H_{\infty},\quad \Lambda_k=f^k(\Lambda_0)\setminus \bigcup_{0\le j\le k-1}\Lambda_j, \quad k\ge 1.
$$
Put $C=\min_{z\in M}m\left(Df|_{E(z)}\right)$, and then define
$$
\delta(x)=\sum_{k=0}^{\infty}\mathbbm{1}_{\Lambda_k}(x)C^kr, \quad T(x)=\sum_{k=0}^{\infty}\mathbbm{1}_{\Lambda_k}(x) \chi^{-k}C^{-k},
$$
where $\mathbbm{1}_A(x)$ denotes the indicator function of the set $A$.
We see from definition that both $\delta(x)$ and $T(x)$ are Borel functions constant on each $\Lambda_k$.
For each $x\in \Lambda_k$, since $f^{-k}(x)\in \Lambda_0=H_{\infty}$, which admits the local   unstable manifold 
$W_r^{E,u}(f^{-k}(x))$ of size $r=\delta(f^{-k}(x))$. By a simple computation, we know $f^k(W_r^{E,u}(f^{-k}(x)))$ contains a ball of size $\delta(x)$ around $x$, denoted as $W^{E,u}_{\delta(x)}(x)$. Moreover, we have 
$$
d(f^{-n}(y), f^{-n}(z))\le T(x)\chi^{-n} d (y, z),\quad \forall \, n\in\NN,
$$
provided that $y,z\in W^{E,u}_{\delta(x)}(x)$. By Claim \ref{theta}, we know $\mu(\Lambda_0)=\mu(H_{\infty})\ge \theta$. Applying the ergodicity of $\mu$, we get
$$
\mu\left(\bigcup_{k\ge 0}\Lambda_k\right)=\mu\left(\bigcup_{k\ge 0}f^k(\Lambda_0)\right)=1.
$$
Therefore, we have shown that for $\mu$-almost every point $x\in M$, there exists a local   unstable manifold $W^{E,u}_{\delta(x)}(x)$. 

For $\mu$-almost every point $x\in M$, define
$$
W^{E,u}(x)=\bigcup_{n\ge 0}f^n\left(W^{E,u}_{\delta(f^{-n}(x))}(f^{-n}(x))\right).
$$
Now we show it is a global   unstable manifold at $x$. Observe first by construction that for every $y,z\in W^{E,u}(x)$, it holds that
$$
\limsup_{n\to +\infty}\frac{1}{n} \log d(f^{-n}(y), f^{-n}(z))\le -\log \chi<0.
$$
Thus it suffices to show it is an injectively immersed submanifold. 
Since $\mu(\Lambda_0)\ge \theta>0$, the ergodicity of $\mu$ implies that for $\mu$-almost every $x$, its backward iterates  $\{f^{-n}(x)\}$ enters $\Lambda_0$ infinitely many times, whose local   unstable manifolds have size $r$.  
Since the size of the backward iterates of local   unstable manifolds tends to zero for almost every point, for every $m\in \NN$ there exists $n\in \NN$ large enough such that $f^{-n}(x)\in \Lambda_0$ and 
$$
f^{-(n-k)}\left(W^{E,u}_{\delta(f^{-k}(x))}(f^{-k}(x))\right)\subset W_r^{E,u}(f^{-n}(x)), \quad 0\le k \le m.
$$
Thus, the union
$$
\bigcup_{0\le k\le m}f^k\left(W^{E,u}_{\delta(f^{-k}(x))}(f^{-k}(x))\right)
$$ 
is contained in the embedded manifold $f^n(W^{E,u}_{r}(f^{-n}(x)))$.
This indicates that $W^{E,u}(x)$ is an injectively immersed submanifold, which is a global   unstable manifold at $x$.
From the construction we know $W^{E,u}(x)$ is tangent to $E$.

\end{proof}

%
%
%
%
%
%
%

\section{Proof of Theorem \ref{m}}\label{main}

Let $f\in {\rm Diff}^1(M)$. For a $Df$-invariant sub-bundle $E\subset TM$ and an invariant measure $\mu$, for $\mu$-almost every $x\in M$, denote the minimal and maximal Lyapunov exponents along $E$ by 
$$
\chi_E^-(x)=\lim_{n\to +\infty}\frac{1}{n}\log m(Df^n|_{E(x)}), \quad \chi^+_E(x)=\lim_{n\to +\infty}\frac{1}{n}\log\|Df^n|_{E(x)}\|.
$$
Note that when $\mu$ is ergodic, $\chi_E^-(x)$ and $\chi_E^+(x)$ are constants at $\mu$-almost every $x\in M$, denoted them by $\chi_E^-(\mu)$ and $\chi_E^+(\mu)$, respectively.

Now we present several results on the existence of   (un)stable manifolds in different settings. We mention that due to the ergodic decomposition theorem \cite[Theorem 6.4]{Man87}, it suffices to consider their ergodic cases. 


\begin{theorem}\label{ra}
Let $f\in {\rm Diff}^1(M)$ with a continuous invariant splitting $TM=E\oplus F$. If $\mu$ is an ergodic measure such that
$$
\chi_E^-(\mu)>\max\{0, ~\chi_F^+(\mu)\},
$$
then at $\mu$-almost every point $x$ there exists an   unstable manifold $W^{E,u}(x)$ tangent to $E$. 

Moreover, for every $0<\varepsilon<\min\{\chi_E^-(\mu), \chi_E^-(\mu)-\chi_F^+(\mu)\}$ there exist Borel functions $\delta(x)$ and $T(x)$ such that for $\mu$-almost every $x\in M$, there is a local   unstable manifold $W^{E,u}_{\delta(x)}(x)$ of size $\delta(x)$ contained in the global   unstable manifold $W^{E,u}(x)$ such that for any $y\in W^{E,u}_{\delta(x)}(x)$ and $n\in \NN$,
$$
d(f^{-n}(x), f^{-n}(y))\le T(x){\rm e}^{-\chi_E^-(\mu)n}{\rm e}^{\varepsilon n}d(x,y).
$$
\end{theorem}

%

Notice that in Theorem \ref{ra}, we do not assume any hyperbolicity of $\mu$ along the sub-bundle $F$. 
By applying Theorem \ref{ra} to $f^{-1}$, one can conclude the existence of stable manifolds as follows:

\begin{theorem}\label{rb}
Let $f\in {\rm Diff}^1(M)$ with a continuous invariant splitting $TM=E\oplus F$. If $\mu$ is an ergodic measure such that
$$
\chi_F^+(\mu)<\min\{0, ~\chi_E^-(\mu)\},
$$
then at $\mu$-almost every point $x$ there exists a   stable manifold $W^{F,s}(x)$ tangent to $F$. 

Moreover, for every $0<\varepsilon<\min\{-\chi_F^+(\mu), \chi_E^-(\mu)-\chi_F^+(\mu)\}$ there exist Borel functions $\delta(x)$ and $T(x)$ such that for $\mu$-almost every $x\in M$, there is a local   stable manifold $W^{F,s}_{\delta(x)}(x)$ of size $\delta(x)$ contained in the global   stable manifold $W^{F,s}(x)$ such that for any $y\in W^{F,s}_{\delta(x)}(x)$ and $n\in \NN$,
$$
d(f^{n}(x), f^{n}(y))\le T(x){\rm e}^{\chi_F^+(\mu)n}{\rm e}^{\varepsilon n}d(x,y).
$$
\end{theorem}

Theorem \ref{m} is an immediate consequence of Theorem \ref{ra} and Theorem \ref{rb}.

We can also establish   (un)stable manifolds for $C^1$ diffeomorphisms possessing continuous invariant  splitting on tangent bundle with more than two sub-bundles. Here we provide the case of three sub-bundles, the more general case can be established similarly.

\begin{theorem}\label{rc}
Let $f\in {\rm Diff}^1(M)$ with a continuous invariant splitting 
$$
TM=E\oplus F\oplus G.
$$
Let $\mu$ be an ergodic measure whose Lyapunov exponents along these three sub-bundles can be separated, i.e.,
$$
\chi_{E}^{-}(\mu)>\chi_{F}^+(\mu)\ge \chi_{F}^{-}(\mu)>\chi_{G}^+(\mu).
$$

If $\chi_F^{-}(\mu)>0$, then at $\mu$-almost every $x$ there exist an   unstable manifold $W^{E\oplus F,u}(x)$ tangent to $E\oplus F$ and an   unstable manifold $W^{E,u}(x)\subset W^{E\oplus F,u}(x)$ tangent to $E$. Likewise, if $\chi_F^{+}(\mu)<0$, then at $\mu$-almost every $x$ there exists a   stable manifold $W^{F\oplus G,s}(x)$ tangent to $F\oplus G$ and an   unstable manifold $W^{G,s}(x)\subset W^{F\oplus G,s}(x)$ tangent to $G$.
\end{theorem}

The proofs of Theorem \ref{ra} and  Theorem \ref{rc} are presented in Subsection \ref{3.1}.

\medskip

The classical Pesin's stable manifold theorem asserts that the size of local Pesin  (un)stable manifolds vary sub-exponentially along orbits. This property was also established for local (un)stable manifolds constructed by Abdenur-Bonatti-Crovisier for $C^1$ diffeomorphisms with dominated splitting (\cite[Proposition 8.9]{ABC}). Thus, a fundamental question arises:

\begin{question}
Do the local   (un)stable manifolds given by Theorems \ref{ra}, \ref{rb}, and \ref{rc} exhibit sub-exponential growth in size along orbits?
\end{question}

To address the above problem, we seem to require a more nuanced discussion than the approach presented in this paper.

\subsection{High density on hyperbolic times with averaged domination}\label{be}

Theorem \ref{mg} demonstrates that,  in constructing unstable manifolds,
the averaged domination property at hyperbolic times can compensate for the absence of uniform domination. We will show below that the gap and positivity  of Lyapunov exponents can give rise to this property. 

As a first stage,  we investigate the abundance of hyperbolic times and the additivity of sub-additive sequences by increasing iterations.

Let $\Phi=\{\varphi_n\}_{n\in \NN}$ be a sequence of continuous functions that is sub-additive, i.e., for any $x\in M$, 
$$
\varphi_{n+m}(x)\le \varphi_n(x)+\varphi_{m}(f^n(x)),\quad \forall \, n,m\in \NN.
$$
Given $\gamma\in \RR$ and $\ell\in \NN$, for $x\in M$, we consider the set consisting of its hyperbolic times with respect to $\varphi_{\ell}$ and $\gamma$ as follows
$$
T_{\ell}(\Phi, x, \gamma)=\{n :\frac{1}{k\ell}\sum_{i=n-k}^{n-1}\varphi_{\ell}(f^{i\ell}(x))\le \gamma,\quad \forall\, 1\le k \le n\}.
$$

We have the following result, which asserts the abundance of hyperbolic times by considering large iteration.

\begin{lemma}\label{extd}
Let $\gamma_1<\gamma_2$ and $\Phi=\{\varphi_n\}_{n\in \NN}$ be a sequence of sub-additive continuous functions.  If $\mu$ is an invariant measure such that
$$
\lim_{n\to +\infty}\frac{\varphi_n(x)}{n}\le \gamma_1, \quad \mu\textrm{-a.e.}~x,
$$
then for any $\theta\in (0,1)$, there exists $\ell_0\in \NN$ such that for any $\ell\ge \ell_0$, we have 
$$
\mu\left(\{x: D_L(T_{\ell}(\Phi, x, \gamma_2))\ge \theta\}\right)>\theta.
$$
\end{lemma}

To prove Lemma \ref{extd}, we recall the next Pliss-Like Lemma given in \cite[Lemma A]{AV17}.

\begin{lemma}\label{plav}
Let $\eta<\zeta<L$. Then for any $\theta\in (0,1)$, there exists $\rho=\rho(L,\eta,\zeta,\theta)\in (0,1)$ such that if $\{a_i\}_{i=0}^{N-1}$ is a sequence of real numbers satisfying $a_i\le L$ for every $0\le i \le N-1$, and
$$
\frac{1}{N}\#\{0\le i \le N-1: a_i<\eta\}>\rho,
$$
then there exist $1\le n_1<n_1<\cdots <n_m\le N-1$ with $m\ge \theta N$ such that
$$
\frac{1}{k}\sum_{j=n_i-k}^{n_i-1}a_j<\zeta, \quad \forall \, 1\le k \le n_i, ~\quad \forall\, 1\le i \le m.
$$
\end{lemma}

Now we can give the proof of Lemma \ref{extd} as follows.

\begin{proof}[Proof of Lemma \ref{extd}]
Let 
$$
L=\max_{x\in M}\{\varphi_1(x)\},\quad \eta=\frac{1}{2}(\gamma_1+\gamma_2), \quad \zeta=\gamma_2.
$$
For any $\theta\in (0,1)$, let $\rho=\rho(L,\eta,\zeta, \theta)\in (0,1)$ be the constant given by Lemma \ref{plav}.

Define for each $\ell\in \NN$ the set
$$
B_{\ell}=\{x: \varphi_{\ell}(x)<\ell \eta\}.
$$
By assumption,
$$
\lim_{n\to +\infty}\frac{\varphi_n(x)}{n}\le \gamma_1<\eta<\gamma_2
$$
for $\mu$-almost every $x\in M$. Thus
there exists $\ell_0\in \NN$ such that 
\begin{equation}\label{first}
\mu(B_{\ell})>1-(1-\rho)(1- \theta), \quad \forall \,\ell\ge \ell_0.
\end{equation}
Now we fix any $\ell \ge \ell_0$. By Birkhoff's ergodic theorem we know that the limit
$$
\psi(x):=\lim_{n\to +\infty} \frac{1}{n}\sum_{j=0}^{n-1}\mathbbm{1}_{B_{\ell}}\left(f^{j\ell}(x)\right)
$$
exists for $x$ in a full $\mu$-measure subset $K$. Furthermore, together with (\ref{first}) we have 
\begin{equation}\label{sec}
\int \psi \,{\rm d}\mu=\mu(B_{\ell})>1-(1-\rho)(1-\theta).
\end{equation}
As a result, 
\begin{eqnarray*}
\mu(\{x\in K: 1-\psi(x)\ge 1-\rho\})&\le& \frac{1}{1-\rho}\int (1-\psi)\, {\rm d}\mu\\
&<& \frac{(1-\rho)(1-\theta)}{1-\rho}\\
&=&1-\theta.
\end{eqnarray*}
Let $X:=\{x\in K: \psi(x)>\rho\}$, we then get $\mu(X)>\theta$.
For any $x\in X$, take $n_x\in \NN$ such that 
\begin{equation}\label{rhob}
\frac{1}{N}\#\{0\le i \le N-1: f^{i\ell}(x)\in B_{\ell}\}>\rho,\quad \forall \, N\ge n_x.
\end{equation}
Let
$$
a_i=\frac{\varphi_{\ell}(f^{i\ell}(x))}{\ell},\quad 0\le i \le N-1.
$$
Then, with the definition of $B_{\ell}$, \eqref{rhob} suggests that 
$$
\frac{1}{N}\#\{0\le i \le N-1: a_i<\eta\}>\rho.
$$
By the sub-additivity and the choice of $L$, we have for every $0\le i \le N-1$ that
$$
a_i\le \frac{1}{\ell}\sum_{j=0}^{\ell-1}\varphi_1(f^{i\ell+j}(x))\le L.
$$
Therefore, by applying Lemma \ref{plav} one knows that 
there exist $1\le n_1<n_1<\cdots <n_m\le N-1$ with $m\ge \theta N$ such that
$$
\frac{1}{k\ell}\sum_{i=n_i-k}^{n_i-1}\varphi_{\ell}(f^{i\ell}(x))<\gamma_2, \quad \forall \, 1\le k \le n_i, ~\quad \forall\, 1\le i \le m.
$$
Hence, 
$$
\#(T_{\ell}(\Phi, x,\gamma_2)\cap \{0,\cdots,N-1\})\ge m \ge \theta N.
$$
Note that this holds for every $N\ge n_x$, and $\theta$ is independent of $N$, one concludes that 
$$
\liminf_{n\to +\infty}\frac{1}{n}\#\left(\{0,\cdots, n-1\}\cap T_{\ell}(\Phi, x, \gamma_2)\right)\ge \theta, \quad \forall \, x\in X.
$$
This completes the proof. 
\end{proof}

\medskip
Given $\gamma\in \RR$ and $\ell\in \NN$, consider the block defined as follows 
$$
H_{\ell}(\Phi,\gamma)=\{x: \frac{1}{n\ell}\sum_{i=0}^{n-1}\varphi_{\ell}(f^{i\ell}(x))\le \gamma,\quad \forall\, n\in \NN\}.
$$
We have the following result, which shows the additivity from sub-additivity, and is crucial for getting averaged domination. We mention that it has been appeared in previous works under more specific settings, such as \cite[Lemma A.2]{MC21} and \cite[Lemma 2.7]{CZZ23}. 

%

\begin{lemma}\label{toblock}
Let $\gamma_1<\gamma_2$ and $\Phi=\{\varphi_n\}_{n\in \NN}$ be a sequence of sub-additive continuous functions.  If $\mu$ is an invariant measure such that
$$
\lim_{n\to +\infty}\frac{\varphi_n(x)}{n}\le \gamma_1, \quad \mu\textrm{-a.e.}~x,
$$
then 
$$
\lim_{\ell \to +\infty}\mu(H_{\ell}(\Phi, \gamma_2))=1.
$$
\end{lemma}
\medskip
Now we will apply Lemma \ref{extd} and Lemma \ref{toblock} to $C^1$ diffeomorphism $f$ with two continuous invariant sub-bundles $E$ and $F$. 
Given $\ell\in \NN$, $\gamma_1, \gamma_2\in \RR$, denote 
\begin{align*}
& \quad \Lambda_{\ell}(\gamma_1,\gamma_2,E,F)=\Big\{x: \frac{1}{n\ell}\sum_{i=0}^{n-1}\log m\left(Df^{\ell}|_{E(f^{i\ell}(x))}\right)\ge \gamma_1,\\
 &\quad \quad \quad \quad \quad\quad\quad\quad~~\frac{1}{n\ell}\sum_{i=0}^{n-1}\log\|Df^{\ell}|_{F(f^{i\ell}(x))}\|\le \gamma_2, \quad \forall\, n\in \NN\Big\}.
\end{align*}

When associated to Lyapunov exponents, $\{\Lambda_{\ell}(\gamma_1,\gamma_2,E,F)\}_{\ell\in \NN}$ plays a role like ``Pesin blocks" suggested as follows. 

\begin{proposition}\label{tpos}
Let $f\in {\rm Diff}^1(M)$ with two continuous invariant sub-bundles $E$ and $F$, and $\mu$ be an ergodic measure. If $\gamma_1<\chi_E^-(\mu)$ and  $\gamma_2>\chi_F^+(\mu)$, then 
$$
\lim_{\ell\to +\infty}\mu(\Lambda_{\ell}(\gamma_1,\gamma_2,E,F))=1.
$$
\end{proposition}

\begin{proof}
Let us fix constants as in the assumption. For any $\ell\in \NN$, consider the subset defined by
$$
\Lambda_{\ell}^s(\gamma_2)=\{x: \frac{1}{n\ell}\sum_{i=0}^{n-1}\log\|Df^{\ell}|_{F(f^{i\ell}(x))}\|\le \gamma_2, \quad \forall\, n\in \NN\}.
$$
Now we show 
\begin{equation}\label{k}
 \lim_{\ell\to +\infty}\mu(\Lambda_{\ell}^s(\gamma_2))=1.
\end{equation}
Let 
$$
\varphi_n(x)=\log\|Df^{n}|_{F(x)}\|, \quad n\in \NN, \quad x\in M.
$$
We see from chain's rule that $\Phi=\{\varphi_n\}_{n\in \NN} $ is a sub-additive continuous sequence. 
Since 
$$
\lim_{n\to +\infty}\frac{\varphi_n(x)}{n}=\chi_F^+(\mu)<\gamma_2, 
$$
one can apply Lemma \ref{extd} to get that
$
\mu(H_{\ell}(\Phi,\gamma_2))\to 1
$
as $\ell \to +\infty$.
Note that $\Lambda_{\ell}^s(\gamma_2)=H_{\ell}(\Phi,\gamma_2)$, we arrive at the convergence (\ref{k}).

Now take $\zeta_n(x)=-\log m(Df^n|_{E(x)})$ for every $x\in M$ and $n\in \NN$. We know that $\{\zeta_n\}_{n\in \NN}$ is also a sub-additive continuous sequence, which satisfies
$$
\lim_{n\to +\infty}\frac{\zeta_n(x)}{n}=-\lim_{n\to +\infty}\frac{1}{n}\log m\left(Df^{n}|_{E(x)}\right)=-\chi_E^-(\mu)<-\gamma_1.
$$
By applying Lemma \ref{extd} once more to $\{\zeta_n\}_{n\in \NN}$, as above we obtain
\begin{equation}\label{jhj}
\lim_{\ell\to +\infty}\mu(\Lambda_{\ell}^u(\gamma_1))=1,
\end{equation}
where
$$
\Lambda_{\ell}^u(\gamma_1)=\{x: \frac{1}{n\ell}\sum_{i=0}^{n-1}\log m\left(Df^{\ell}|_{E(f^{i\ell}(x))}\right)\ge \gamma_1, \quad\forall \,n\in \NN\}.
$$
Combining (\ref{k}) and (\ref{jhj}), with the
observation $\Lambda_{\ell}(\gamma_1,\gamma_2,E,F)=\Lambda^u_{\ell}(\gamma_1)\cap \Lambda_{\ell}^s(\gamma_2)$,  we obtain
$$
\lim_{\ell\to+\infty}\mu(\Lambda_{\ell}(\gamma_1,\gamma_2,E,F))=1,
$$
which completes the proof. 

\end{proof}



Recalling the definition of $\Delta$ from \eqref{def:delta}, we have the following observation: for any $\ell\in\NN$, if $\gamma_1>\gamma_2$ then 
\begin{equation}\label{bta}
\Lambda_{\ell}(\gamma_1,\gamma_2,E,F)\subset \D({\rm e}^{(\gamma_1-\gamma_2)\ell}, f^{\ell},E,F).
\end{equation}
Therefore, Proposition \ref{tpos} suggests that the averaged domination may arise from the gap between Lyapunov exponents of sub-bundles. More precisely, we have
%
%
%
%
\begin{corollary}\label{td}
Let $f\in {\rm Diff}^1(M)$ with two continuous invariant sub-bundles $E$ and $F$.
If $\mu$ is an ergodic measure such that $\chi_E^-(\mu)>\chi_F^+(\mu)$, then 
$$
\lim_{\ell\to +\infty}\mu(\D({\rm e}^{(\gamma_1-\gamma_2)\ell},f^{\ell},E,F))=1
$$
whenever $\chi_E^-(\mu)>\gamma_1>\gamma_2>\chi_F^+(\mu)$.
\end{corollary}

We will concentrate our discussion on points whose orbits admit averaged domination with high density on hyperbolic times. For $\gamma_1>\max\{0,\gamma_2\}$ and $\theta\in (0,1)$, we denote by $\Lambda_{\ell}(\gamma_1,\gamma_2,\theta,E,F)$ the set of points $x$ satisfying
$$
D_L(\mathcal{HD}\left(x,{\rm e}^ {\gamma_1 \ell}, {\rm e}^{(\gamma_1-\gamma_2)\ell},f^{\ell},E,F\right))\ge \theta,
$$
where $\mathcal{HD}$ is defined in Definition \ref{hdx}.



\begin{proposition}\label{sah}
Let $f\in {\rm Diff}^1(M)$ with two continuous invariant sub-bundles $E$ and $F$,
and $\mu$ be an ergodic measure satisfying $\chi_E^-(\mu)>\max\{0,\chi_F^+(\mu)\}$. Given any constants $\gamma_1,\gamma_2\in \RR$ such that
$$
\chi_E^{-}(\mu)>\gamma_1>\gamma_2>\chi_F^+(\mu)\quad \textrm{and}\quad \gamma_1>0,
$$
then for any $\theta\in (0,1)$, there exists $\ell_0\in \NN$ such that
$$
\mu(\Lambda_{\ell}(\gamma_1,\gamma_2,\theta,E,F))>\theta,\quad \forall \, \ell \ge \ell_0.
$$
\end{proposition}

\begin{proof}
Let
$$
\Phi=\{\varphi_n\}_{n\in \NN},\quad \varphi_n=-\log m(Df^n|_{E}),\quad n\in \NN.
$$
Then, $\Phi$ is a sequence of sub-additive continuous functions. Moreover, we have 
by assumption that
$$
\lim_{n\to +\infty}\frac{\varphi_n(x)}{n}=-\lim_{n\to +\infty}\frac{1}{n}\log m(Df^n|_{E(x)})=-\chi_E^{-}(\mu)<-\gamma_1
$$
for $\mu$-almost every $x\in M$. For any $\theta\in (0,1)$, by applying Lemma \ref{extd} to $\Phi$, there exists $\ell_0\in \NN$ such that 
\begin{equation}\label{oneden}
\mu\left(\{x: D_L(T_{\ell}(\Phi, x, -\gamma_1))\ge \frac{\theta+1}{2}>\theta\}\right)>\frac{\theta+1}{2},\quad \forall \, \ell\ge \ell_0.
\end{equation}
On the other hand, according to Proposition \ref{tpos}, up to increasing $\ell_0$, we have 
\begin{equation}\label{twobd}
\mu(\Lambda_{\ell}(\gamma_1,\gamma_2,E,F))>\frac{\theta+1}{2},\quad \forall \, \ell\ge \ell_0.
\end{equation}
Note that if $n\in T_{\ell}(\Phi, x, -\gamma_1)$, then
$$
\frac{1}{k\ell}\sum_{i=n-k}^{n-1}\log m(Df^{\ell}|_{E(f^{i\ell}(x))})\ge \gamma_1,\quad \forall \, 1\le k \le n.
$$
This means that $n$ is a $({\rm e}^{\gamma\ell},E)$-hyperbolic time (recall Definition \ref{hdx}) for $x$ with respect to $f^{\ell}$. 
Combined with the fact \eqref{bta}, we deduce that 
$$
\{x: D_L(T_{\ell}(\Phi, x, -\gamma_1))>\theta\}\cap \Lambda_{\ell}(\gamma_1,\gamma_2,E,F)\subset \Lambda_{\ell}(\gamma_1,\gamma_2,\theta,E,F).
$$
This together with \eqref{oneden} and \eqref{twobd} yields 
$$
\mu(\Lambda_{\ell}(\gamma_1,\gamma_2,\theta,E,F))>\theta, \quad \forall \, \ell\ge \ell_0.
$$
Now we complete the proof.
\end{proof}

\subsection{Proofs of Theorem \ref{ra} and Theorem \ref{rc} }\label{3.1}

Let us give the proof of Theorem \ref{ra} by applying Theorem \ref{mg} and Proposition \ref{sah}.

\begin{proof}[Proof of Theorem \ref{ra}]

Take any $f$-ergodic measure $\mu$ such that 
$$
\chi_E^-(\mu)>\max\{0, ~\chi_F^+(\mu)\}.
$$
For any $0<\varepsilon<\min\{\chi_E^-(\mu), \chi_E^-(\mu)-\chi_F^+(\mu)\}$, we choose 
$$
\gamma_1=\chi_E^-(\mu)-\varepsilon/2,\quad \gamma_2\in (\chi_F^+(\mu),\gamma_1).
$$
Thus $\gamma_1>0$.
By Proposition \ref{sah}, one can fix $\theta>0$ and $\ell\in \NN$ such that 
$$
\mu(\Lambda_{\ell}(\gamma_1,\gamma_2,\theta,E,F))>\theta>0.
$$ 
Note that $\mu$ may not be $f^{\ell}$-ergodic. However, there exists some $f^{\ell}$-ergodic measure $\nu$ such that 
$$
\mu=\frac{1}{\ell}\left(\nu+f_{\ast}\nu+\cdots +f_{\ast}^{\ell-1}\nu\right).
$$
Since $\mu(\Lambda_{\ell}(\gamma_1,\gamma_2,\theta,E,F))>0$, up to considering another $f_{\ast}^i\nu$, we can assume
$$
\nu(\Lambda_{\ell}(\gamma_1,\gamma_2,\theta,E,F))>\theta>0.
$$ 
Take a $\nu$-generic point $x\in\Lambda_{\ell}(\gamma_1,\gamma_2,\theta,E,F)$ with respect to $f^{\ell}$. Then we have
$$
D_U(\mathcal{HD}(x,\lambda_1,\lambda_2, f^{\ell},E,F))\ge D_L(\mathcal{HD}(x,\lambda_1,\lambda_2, f^{\ell},E,F))\ge \theta,
$$
where 
$$
\lambda_1={\rm e}^{(\chi_E^-(\mu)-\varepsilon/2) \ell}>1,\quad \lambda_2={\rm e}^{(\gamma_1-\gamma_2)\ell}>1.
$$
Therefore, $(f^{\ell},\nu,x)$ satisfies the assumption of Theorem \ref{mg}.
Take 
$$
\chi={\rm e}^{(\chi_E^-(\mu)-\varepsilon) \ell}.
$$
Then, we have $1<\chi<\lambda_1$ by construction.
%
Hence, there are Borel functions $\delta_1(x)$ and $T_1(x)$ such that for $\nu$-almost every $x\in M$, there is a local   unstable manifold $W^{E,u}_{\delta_1(x)}(x,f^{\ell})\subset W^{E,u}(x,f^{\ell})$ such that for any $y\in W^{E,u}_{\delta_1(x)}(x,f^{\ell})$, we have
$$
d\left(f^{-\ell n}(x),f^{-\ell n}(y)\right)\le T_1(x) \chi^{- n}d(x,y)
$$
for every $n\in \NN$.

Going back to $f$, and using the ergodicity of $\mu$, one can find new Borel functions $\delta(x)$ and $T(x)$ so that for $\mu$-almost $x\in M$, there exists a local   unstable manifold $W_{\delta(x)}^{E,u}(x,f)$ contained in its global   unstable manifold $W^{E,u}(x,f)$, and for any $y\in W_{\delta(x)}^{E,u}(x,f)$,
$$
d(f^{-n}(x),f^{-n}(y))\le T(x){\rm e}^{-\chi_E^-(\mu)n}{\rm e}^{\varepsilon n}d(x,y),~\forall \, n\in \NN.
$$
We get the desired result.
\end{proof}



We now prove Theorem \ref{rc}. Unlike the previous case, this requires constructing lower-dimensional   stable and unstable manifolds embedded within higher-dimensional ones.

\begin{proof}[Proof of Theorem \ref{rc}]
We only deal with the case of $\chi_F^{-}(\mu)>0$. With the same argument we can deduce the result of the case $\chi_F^{+}(\mu)<0$. The idea is to establish $W^{E\oplus F,u}$ and $W^{E,u}$ simultaneously, take advantage of the high density on hyperbolic times within averaged domination provided in Proposition \ref{sah}.

%
%
%
%
%
For $\varepsilon$ sufficiently small,
we  put
$$
\alpha_1=\chi_E^{-}(\mu)-\varepsilon,\quad \alpha_2=\chi_F^{+}(\mu)+\varepsilon,
$$
$$
\beta_1=\chi_{E\oplus F}^{-}(\mu)-\varepsilon,\quad \beta_2=\chi_G^{+}(\mu)+\varepsilon.
$$
For $\ell \in\NN$, we take
$$
\lambda_1={\rm e}^{\alpha_1\ell},\quad \lambda_2={\rm e}^{(\alpha_1-\alpha_2)\ell},\quad \kappa_1={\rm e}^{\beta_1\ell},\quad \kappa_2={\rm e}^{(\beta_1-\beta_2)\ell},
$$
$$
\Lambda_{\ell}=\Lambda_{\ell}(\alpha_1,\alpha_2,2/3, E, F)\cap \Lambda_{\ell}(\alpha_1,\alpha_2,2/3, E\oplus F,G).
$$
By applying Proposition \ref{sah} to 
$(E,F)$ and $(E\oplus F, G)$, there exists $\ell\in \NN$ such that 
$
\mu(\Lambda_{\ell})>0.
$
By construction, for every $x\in \Lambda_{\ell}$ we have
$$
D_L(\mathcal{HD}\left(x,\lambda_1,\lambda_2,f^{\ell},E,F\right))\ge 2/3,
$$
and
$$
D_L(\mathcal{HD}\left(x,\kappa_1,\kappa_2,f^{\ell},E\oplus F, G\right))\ge 2/3.
$$
As a consequence, we obtain
$$
D_L(\mathcal{HD}(x))>0,\quad x\in \Lambda_{\ell},
$$
by taking
$$
\mathcal{HD}(x)=\mathcal{HD}\left(x,\lambda_1,\lambda_2,f^{\ell},E,F\right)\cap \mathcal{HD}\left(x,\kappa_1,\kappa_2,f^{\ell},E\oplus F, G\right).
$$
Since $\mu(\Lambda_{\ell})>0$, as shown in the proof of Theorem \ref{ra} we can choose an ergodic component $\nu$ of $\mu$ with respect to $g=f^{\ell}$ such that $\mu=\frac{1}{n}\sum_{i=0}^{\ell-1}f_{\ast}^i\nu$ and
$
\nu(\Lambda_{\ell})>0.
$

Fix a $\nu$-generic point $x_0\in \Lambda_{\ell}$. For $\epsilon>0$ small enough, denote by $B^{E\oplus F}(0,\epsilon)$ and $B^E(0,\epsilon)$ the balls centered at the origin of size $\epsilon$ in $E(x)\oplus F(x)$ and $E(x)$ respectively. We take exponential images of these two balls as follows: 
$$
\widehat{D}={\rm exp}_{x_0}(B^{E\oplus F}(0,\epsilon)),\quad D={\rm exp}_{x_0}(B^E(0,\epsilon)).
$$
Then for any $a$ sufficiently small, one can choose $\epsilon$ small enough such that $\widehat{D}$ is tangent to $\mathscr{C}_a^{E\oplus F}$, and for every $x\in D$ and $v\in T_xD$, we have
$$
\|v_F\|\le a\|v_E\|,\quad v=v_E+ v_F+v_G\in E(x)\oplus F(x)\oplus G(x).
$$
By this construction, we see $\widehat{D}$ is close to $E\oplus F$ and $D$ is close to $E$ as long as $a$ is small enough.

Observe that any $n\in \mathcal{HD}(x_0)$ is a $(\kappa_1, E\oplus F)$-hyperbolic time for $x_0$, and 
$(x_0, g^{n}(x_0))$ is $\kappa_2$-average dominated associated to $(E\oplus F, G)$. Thus, proceeding as in the proof of Theorem \ref{mg}, there is a subset $H_{\infty}\subset M$ satisfying $\nu(H_{\infty})>0$ and for any $x\in H_{\infty}$, there exists a subsequence $\{n_i^x\}\subset \mathcal{HD}(x_0)$ and disks $\widehat{D}_{n_i^x}\subset g^{n_i^x}(\widehat{D})$ of uniform size such that 
\begin{itemize}
\item $$g^{n_i^x}(x_0)\to x.$$ 
\smallskip
\item $\widehat{D}_{n_i^x}$ converges to the local   unstable manifold $W^{E\oplus F,u}_{loc}(x,g)$ tangent to $E\oplus F$, as $i\to +\infty$.
\end{itemize}
As any $n\in \mathcal{HD}(x_0)$ is also a $(\lambda_1,E)$-hyperbolic time and
$(x_0, g^{n}(x_0))$ is $\lambda_2$-average dominated associated to $(E,F)$, which implies the backward contraction of $g^{n}(D)$ on size and the distance between $E$, as stated in Lemma \ref{aver}. Consequently, for every $x\in H_{\infty}$, up to considering subsequences of $\{n_i^x\}$ given above, we can find disks $D_{n_i^x} \subset \widehat{D}_{n_i^x}$ that converges to the local  unstable manifold $W^{E,u}_{loc}(x,g)$ tangent to $E$. Together with the convergence $\widehat{D}_{n_i^x}\to W^{E\oplus F,u}_{loc}(x,g)$, we see $W^{E,u}_{loc}(x,g)\subset W^{E\oplus F,u}_{loc}(x,g)$. By the ergodicity of $\nu$ with respect to $f^{\ell}$, one then obtains these (local)   unstable manifolds within the inclusion property at $\nu$-almost every point.

%

Since $\mu=\frac{1}{n}\sum_{i=0}^{\ell-1}f_{\ast}^i\nu$, going back to $f$ with the ergodicity of $\mu$, one gets that for $\mu$-almost every $x\in M$, there exist local   unstable manifolds $W_{loc}^{E\oplus F,u}(x,f)$ and $W_{loc}^{E,u}(x,f)$, along with their corresponding global   unstable manifolds $W^{E\oplus F,u}(x,f)$ and $W^{E,u}(x,f)$. Moreover, 
$$
W^{E,u}(x,f)\subset W^{E\oplus F,u}(x,f),\quad \mu\textrm{-a.e.}~x.
$$
\end{proof}

\bigskip
{\bf{Acknowledgements.}} We thank the anonymous referees who helped us to improve the presentation of this paper.

\bibliographystyle{amsplain}

\begin{thebibliography}{10}
	
	\bibitem{ABC}
	F.~Abdenur, C.~Bonatti, and S.~Crovisier, \emph{Nonuniform hyperbolicity for
		{$C^1$}-generic diffeomorphisms}, Israel J. Math. \textbf{183} (2011), 1--60.
	
	
	\bibitem{AV17}
	M.~Andersson and C.~V\'{a}squez, \emph{Statistical stability of mostly
		expanding diffeomorphisms}, Ann. Inst. H. Poincar\'{e} C Anal. Non
	Lin\'{e}aire \textbf{37} (2020), no.~6, 1245--1270.
	
	\bibitem{ab12}
	A.~Avila and J.~Bochi, \emph{Nonuniform hyperbolicity, global dominated
		splittings and generic properties of volume-preserving diffeomorphisms},
	Trans. Amer. Math. Soc. \textbf{364} (2012), no.~6, 2883--2907.
	
	\bibitem{BC13}
	C.~Bonatti, S.~Crovisier, and K.~Shinohara, \emph{The {$C^{1+\alpha}$}
		hypothesis in {P}esin {T}heory revisited}, J. Mod. Dyn. \textbf{7} (2013),
	no.~4, 605--618.
	
	\bibitem{BS}
	M.~Brin and G.~Stuck, \emph{Introduction to dynamical systems}, back ed.,
	Cambridge University Press, Cambridge, 2015.
	
	
	\bibitem{CZZ23}
	Y.~Chen, Y.~Zang, and R.~Zou, \emph{Lyapunov irregular set of Banach cocycles}, Nonlinearity. \textbf{36} (2023) 5474--5497.
	
	\bibitem{FY}
	A.~Fathi, M.-R. Herman, and J.-C. Yoccoz, \emph{A proof of {P}esin's stable
		manifold theorem}, Geometric dynamics ({R}io de {J}aneiro, 1981), Lecture
	Notes in Math., vol. 1007, Springer, Berlin, 1983, pp.~177--215.
	
	\bibitem{HPS77}
	M.~W. Hirsch, C.~C. Pugh, and M.~Shub, \emph{Invariant manifolds}, Lecture
	Notes in Mathematics, vol. Vol. 583, Springer-Verlag, Berlin-New York, 1977.
	
	\bibitem{K80}
	A.~Katok, \emph{Lyapunov exponents, entropy and periodic orbits for
		diffeomorphisms}, Inst. Hautes \'{E}tudes Sci. Publ. Math. (1980), no.~51,
	137--173.
	
	\bibitem{M84}
	R.~Ma\~{n}\'{e}, \emph{Oseledec's theorem from the generic viewpoint},
	Proceedings of the {I}nternational {C}ongress of {M}athematicians, {V}ol. 1,
	2 ({W}arsaw, 1983), PWN, Warsaw, 1984, pp.~1269--1276.
	
	\bibitem{Man87}
	\bysame, \emph{Ergodic theory and differentiable dynamics}, vol.~8,
	Springer-Verlag, Berlin, 1987.
	
	\bibitem{MCY18}
	Z.~Mi, Y.~Cao, and D.~Yang, \emph{S{RB} measures for attractors with continuous invariant splittings}, Math. Z. \textbf{288} (2018), no.~1-2, 135--165.
		
	\bibitem{MC21}	
	Z.~Mi, Y.~Cao,  \emph{Statistical stability for diffeomorphisms with mostly expanding and mostly contracting centers}, Math. Z. \textbf{299}(2021), 2519--2560.
	
	\bibitem{O}
	V.~I. Oseledec, \emph{A multiplicative ergodic theorem. {C}haracteristic
		{L}japunov, exponents of dynamical systems}, Trudy Moskov. Mat. Ob\v s\v c. \textbf{19} (1968), 179--210.
	
	\bibitem{P76}
	Y.~Pesin, \emph{Families of invariant manifolds that correspond to nonzero
		characteristic exponents}, Izv. Akad. Nauk SSSR Ser. Mat. \textbf{40} (1976),
	no.~6, 1332--1379, 1440.
	
	\bibitem{P77}
	\bysame, \emph{Characteristic {L}japunov exponents, and smooth ergodic theory},
	Uspehi Mat. Nauk \textbf{32} (1977), no.~4(196), 55--112, 287.
	
	\bibitem{Pliss}
	V.~A. Pliss, \emph{On a conjecture of {S}male}, Differencial{$^{\prime}$}nye
	Uravnenija \textbf{8} (1972), 268--282.
	
	\bibitem{P84}
	C.~Pugh, \emph{The {$C\sp{1+\alpha }$} hypothesis in {P}esin theory}, Inst.
	Hautes \'{E}tudes Sci. Publ. Math. (1984), no.~59, 143--161.
	
	\bibitem{R79}
	D.~Ruelle, \emph{Ergodic theory of differentiable dynamical systems}, Inst.
	Hautes \'{E}tudes Sci. Publ. Math. (1979), no.~50, 27--58.
	
	\bibitem{wa82}
	P.~Walters, \emph{An introduction to ergodic theory}, Graduate Texts in
	Mathematics, vol.~79, Springer-Verlag, New York-Berlin, 1982.
	
\end{thebibliography}

\end{document}